\documentclass[11pt,twoside]{amsart}
\usepackage{todonotes}
\usepackage{hyperref}
\hypersetup{
    colorlinks=true,
    allcolors=gray
}
\usepackage[svgnames]{xcolor}
\usepackage[margin=1.35in]{geometry}
\usepackage{slashed}
\usepackage{amssymb}
\usepackage{amsmath}
\usepackage{amsthm}
\usepackage{mathrsfs}
\usepackage{dutchcal}
\usepackage{color}
\usepackage{wrapfig}
\usepackage{import}
\usepackage{xifthen}
\usepackage{pdfpages}
\usepackage{transparent}
\usepackage{eucal}
\usepackage{tensor}
\usepackage{parskip}
\usepackage{mathtools}
\usepackage{extpfeil}
\usepackage{thmtools}{}
\usepackage{tikz}
\usepackage{tikz-cd}
\usetikzlibrary{arrows.meta}

\tikzset{
    symbol/.style={%
        draw=none,
        every to/.append style={%
            edge node={node [sloped, allow upside down, auto=false]{$#1$}}}
    }
}{}
\usepackage[bbgreekl]{mathbbol}
\usepackage[shortlabels]{enumitem}
\usepackage{setspace}

\DeclareMathOperator{\Ric}{Ric}

\newcommand{\Z}{\mathbb{Z}}
\newcommand{\R}{\mathbb{R}}

\renewcommand{\phi}{\varphi}

\newcommand{\C}{\mathbb{C}}

\renewcommand{\tilde}{\widetilde}

\newcommand{\qand}{\quad\text{and}\quad}

\declaretheorem{theorem}
\declaretheorem[numbered=no,
name=Theorem]{theorem*}
\declaretheorem[numbered=no,
name=Corollary]{corollary*}
\declaretheorem[sibling=theorem]{lemma}
\declaretheorem[sibling=theorem]{proposition}
\declaretheorem[sibling=theorem]{corollary}

\declaretheorem[style=remark,sibling=theorem]{remark}

\declaretheoremstyle[
spaceabove=6pt, spacebelow=6pt,
headfont=\normalfont\bfseries,
notefont=\mdseries, notebraces={(}{)},
bodyfont=\normalfont,
postheadspace=1em
]{definition}
\declaretheorem[style=definition,sibling=theorem]{definition}

\usepackage[T1]{fontenc}
\usepackage{microtype}

\title{An almost K\"ahler Cheeger--Gromoll splitting theorem with applications}
\author{Anthony Nguyen, Shengzhen Ning, and Lauren Pusey-Nazzaro}
\AtEndDocument{{\footnotesize
\textsc{School of Mathematics, University of Minnesota, Minneapolis, MN, US} \par
		\textit{E-mail address}: \texttt{nguy5343@umn.edu} \par
		\addvspace{\medskipamount}
		\textsc{Department of Mathematics, University of Maryland, College Park, MD, US} \par
		\textit{E-mail address}: \texttt{ning0040@umn.edu} \par
		\addvspace{\medskipamount}
		\textsc{School of Mathematics, University of Minnesota, Minneapolis, MN, US} \par
		\textit{E-mail address}: \texttt{pusey005@umn.edu} \par
		
}}
\begin{document}
\begin{abstract}
In this paper, we establish an almost K\"ahler analogue of the Cheeger--Gromoll splitting theorem for complete almost K\"ahler manifolds with nonnegative Ricci curvature. As applications, we use the splitting to obtain Goldberg-type integrability results and establish a relation between symplectic non-hyperbolicity and nonnegative Ricci curvature via a theorem of Bangert.
\end{abstract}
\maketitle

\section{Introduction}\label{sec:intro}
The classical Cheeger--Gromoll splitting theorem
\cite{CGsplitting} states that a line in a complete Riemannian manifold with nonnegative Ricci curvature splits off isometrically. The splitting has a natural refinement in K\"ahler geometry. Indeed, since the complex structure is parallel, the maximal Euclidean factor in the universal cover is $J$-invariant and therefore has even dimension, giving a complex Euclidean factor. This K\"ahler splitting is a fundamental ingredient in the structure theorem of Campana--Demailly--Peternell \cite{CDP15} for compact K\"ahler manifolds with semipositive anticanonical bundle. More recently, Chu--Lee--Zhu \cite{ChuLeeZhu} extended the K\"ahler splitting to manifolds with nonnegative mixed curvature.

In this paper, motivated by the search for symplectic analogues of classical theorems in Riemannian geometry, we investigate the splitting theorem in the almost K\"ahler context. A notable precedent is McDuff's symplectic analogue of the Cartan--Hadamard theorem in the K\"ahler setting \cite{McDuff88}, which was recently extended to the almost K\"ahler setting by Cristofaro-Gardiner \cite{DCG26}. 

An {\bf almost K\"ahler manifold} is a tuple $(X,\omega,J,g)$ where $(X,g)$ is a Riemannian manifold, $J\colon TX \to TX$ is an almost complex structure, and $\omega \in \Omega^2(X)$ is a symplectic form satisfying $g(-,-) = \omega(-,J-)$. 
Let $i\colon T\C \to T\C$ denote the standard complex structure on $\C \cong \R_s \times \R_t$. We call $(X,\omega, J, g)$ {\bf split} if it is isomorphic to 
\[ 
(\C \times X', ds \wedge dt + \omega',i\oplus J', ds^2 + dt^2 + g'),
\]
where $(X',\omega',J',g')$ is an almost K\"ahler manifold with $\dim_{\C}X'=\dim_{\C}X-1$.
Let $\pi \colon \overline X \to X$ denote the universal cover of $X$. The map $\pi$ determines an almost K\"ahler structure on $\overline X$ via pullback:
\[
(\overline X, \overline \omega, \overline J, \overline g):= (\overline X, \pi^* \omega, \pi^*J, \pi^* g).
\]
This choice of almost K\"ahler structure on $\overline X$ will be tacitly understood in what follows.

Here is our main result.
\begin{theorem}\label{thm:main}
     Let $(X,\omega,J,g)$ be an almost K\"ahler manifold such that  $g$ is complete with $\Ric_g\geq 0$. Suppose that
    $(X,g)$ contains a Riemannian line $\gamma\colon \R \to X$.
    Then its universal cover $(\overline X, \overline \omega, \overline J, \overline g)$ splits.
\end{theorem}

\subsection{Automatic integrability}
A famous problem in almost K\"ahler geometry is the Goldberg conjecture \cite{Goldberg1969}, which asserts that every compact almost K\"ahler Einstein manifold is necessarily K\"ahler, hence K\"ahler--Einstein. Sekigawa \cite{Sekigawa} proved the conjecture under the assumption of nonnegative scalar curvature. Further Goldberg-type results have been established under assumptions weaker than the Einstein condition; see \cite{Draghici1994,Draghici1995,Draghici1999,ApostolovDraghici2000,ApostolovDraghiciKotschick1999,Kirchberg} and the survey articles \cite{apostolov_draghici_survey,OguroSekigawa2005} for progress in this direction.


Using the almost K\"ahler splitting \autoref{thm:main}, we determine the four-dimensional symplectic manifolds for which nonnegative Ricci curvature forces integrability of the almost complex structure. We also obtain a complementary Goldberg-type result in all dimensions in which more delicate curvature assumptions are replaced by the topological condition of symplectic asphericity. 
Recall that a symplectic manifold $(X,\omega)$ is \textbf{symplectically aspherical} if for any smooth map $f\colon S^2 \to X$ we have $\int_{S^2}f^*\omega = 0$. 

\begin{theorem}[Automatic integrability]\label{thm:aspherical-flat}
Let $(X,\omega,J,g)$ be a closed almost K\"ahler manifold with
$\Ric_g\geq 0$. If 
\begin{itemize}[nosep]
    \item $(X,\omega)$ is symplectically aspherical, or
    \item $\dim_\R X=4$ and $X$ is not diffeomorphic to a rational surface,
\end{itemize}
 then $J$ must be integrable.
\end{theorem}

 \begin{remark}\label{rmk:perturb}
    The bulleted assumptions in \autoref{thm:aspherical-flat} cannot both be removed. Let $(X,J_0)$ be any del Pezzo surface, so that $X$ is diffeomorphic to either $\mathbb{CP}^1\times\mathbb{CP}^1$ or $\mathbb{CP}^2\#k\overline{\mathbb{CP}}{}^2$ for $0\leq k\leq 8$. Since $c_1(X,J_0)$ can be represented by a positive $(1,1)$-form, Yau's solution of the Calabi conjecture yields a K\"ahler structure $(\omega,J_0,g_0)$ with $\Ric_{g_0}>0$. One can then choose a non-integrable $\omega$-compatible almost complex structure $J$ arbitrarily close to $J_0$. The associated almost K\"ahler metric
    \(
        g_J(\cdot,\cdot):=\omega(\cdot,J\cdot)
    \)
    is arbitrarily close to $g_0$ in the $C^\infty$-topology. Since positive Ricci curvature is an open condition in the $C^2$-topology, $g_J$ still satisfies $\Ric_{g_J}>0$ provided $J$ is chosen sufficiently close to $J_0$. 
\end{remark}

\subsection{Symplectic non-hyperbolicity}
We call a symplectic manifold $(X,\omega)$ {\bf hyperbolic} if for \emph{every} $\omega$-tame almost complex structure $J$, every $J$-holomorphic map $\C \to X$ is constant; we call $(X,\omega)$ {\bf non-hyperbolic} if there exists a nonconstant $J$-holomorphic map $\C \to X$ for each $\omega$-tame $J$. This definition is a symplectic analogue of Brody hyperbolicity in complex geometry (see \autoref{section:hyperbolic}).

Negatively curved manifolds exhibit many flavors of hyperbolicity; see, for example, \cite{kobayashi2005hyperbolic, Royden, Gromovhyperbolic,Demailly_algebraic_1997, Zheng02, ChenYang}. 
On the other hand, as an application of the almost K\"ahler splitting \autoref{thm:main}, we establish a nonnegative curvature criterion for symplectic non-hyperbolicity. Our proof employs methods developed by Bangert \cite{Bangert1998Existence} for proving non-hyperbolicity for the standard symplectic form on $T^{2n}$. 






\begin{corollary}\label{thm:non-hyperbolicity}
Let $(X,\omega,J,g)$ be a closed almost K\"ahler manifold with
\(
    \Ric_g\geq 0.
\)
 If $(X,\omega)$ is symplectically aspherical, then $\omega$ is non-hyperbolic.
\end{corollary}

\begin{remark}
Without the symplectic asphericity assumption, the corresponding symplectic non-hyperbolicity statement is unknown to us. Kamenova--Lu--Verbitsky \cite{KLV14} proved that every integrable complex structure on a K3 surface, including non-projective ones, is non-hyperbolic. This can be used to show that the particular $J$ appearing in the hypotheses of \autoref{thm:non-hyperbolicity} is non-hyperbolic in this case; see \autoref{cor:AKnonhyperbolic}. However, this result does not address arbitrary $\omega$-tame almost complex structures, and therefore does not imply the non-hyperbolicity of $\omega$ in our sense.
\end{remark}

\subsection*{Organization of the paper}
In \autoref{section:Riemannianprelim}, we review the main ingredients in the Riemannian Cheeger--Gromoll splitting theorem. The almost K\"ahler splitting \autoref{thm:main} is proved in \autoref{section:AKsplitting}. We then study the integrability of the almost complex structure in \autoref{section:Kahlerness}, where we prove \autoref{thm:dim4} and \autoref{cor:aspherical_nonnegative_is_kahler-flat}; together, these imply \autoref{thm:aspherical-flat}. Finally, in \autoref{section:hyperbolic}, we establish the symplectic non-hyperbolicity result, \autoref{thm:non-hyperbolicity}.

\subsection*{Acknowledgments} The authors would like to thank Dan Cristofaro-Gardiner, Tian-Jun Li, and Sai-Kee Yeung for their interest in this work. The first author was supported by a GAANN Fellowship. The third author was supported by NSF GRFP award \#2237827.

\section{Some Riemannian preliminaries}\label{section:Riemannianprelim}

Throughout this section, all Riemannian manifolds are assumed to be connected and without boundary.  We review the main ingredients of the Cheeger--Gromoll splitting theorem and some of its consequences that will be used in the sequel.

\subsection{Rays, lines, and Busemann functions}

\begin{definition}
Let $(M,g)$ be a complete Riemannian manifold.
A unit-speed geodesic
$
    \gamma\colon [0,\infty)\longrightarrow M
$
is called a {\bf ray} if
\(
    d\bigl(\gamma(s),\gamma(t)\bigr)=t-s
\)
for every $0\leq s\leq t$.
A unit-speed geodesic
$
    \gamma\colon \mathbb{R}\longrightarrow M
$
is called a {\bf Riemannian line}, or a {\bf line}, if
\(
    d\bigl(\gamma(s),\gamma(t)\bigr)=|t-s|
\)
for every $s,t\in\mathbb{R}$.
\end{definition}

Thus, every segment of a ray or a line is globally length minimizing. 
Every complete non-compact Riemannian manifold admits a ray issuing from each point \cite[Lemma~7.3.1]{Petersen2016}, whereas the existence of a line is a stronger condition and does not hold in general. Nevertheless, the existence of a line is automatic when there is a cocompact isometric group action.

\begin{lemma}
\label{lem:cocompact-line}
Let $(M,g)$ be a complete noncompact Riemannian manifold. Suppose that a group $\Gamma$ acts cocompactly on $M$ by isometries. Then $M$ contains a line. In particular, every noncompact regular Riemannian covering of a compact Riemannian manifold contains a line.
\end{lemma}

\begin{proof}
Fix $p\in M$ and choose $x_i\in M$ such that
\(
    L_i:=d(p,x_i)\to\infty.
\)
Let
\(
    \sigma_i:
    [-\frac{L_i}{2},\frac{L_i}{2}]
    \to M
\)
be a unit-speed minimizing geodesic from $p$ to $x_i$, parametrized so that $\sigma_i(0)$ is its midpoint.
Cocompactness provides a compact set $K\subset M$ such that
\(
    \Gamma K=M.
\)
Choose $\gamma_i\in\Gamma$ with
\(
    \gamma_i\sigma_i(0)\in K.
\)
The unit tangent vectors
\(
    (\gamma_i)_*\dot{\sigma}_i(0)
\)
lie in the unit tangent bundle over the compact set $K$. After passing to a subsequence, they converge to a unit vector $v\in T_qM$ for some $q\in K$.
Let
\(
    \sigma:\mathbb{R}\to M
\)
be the complete geodesic with
\(
    \sigma(0)=q,
    \dot{\sigma}(0)=v.
\)
For every compact interval $[a,b]\subset\mathbb{R}$, the translated
geodesics $\gamma_i\circ\sigma_i$ are defined and minimizing on $[a,b]$ for all sufficiently large $i$. Passing to the limit gives
\(
    d(\sigma(a),\sigma(b))=b-a.
\)
Hence $\sigma$ is a line.
\end{proof}



\begin{definition}
\label{def:busemann}
Let $\gamma\colon [0,\infty)\to M$ be a ray. The \textbf{Busemann function} associated to $\gamma$ is
\[
    b_\gamma(x)
    :=
    \lim_{t\to\infty}
    \bigl(d(x,\gamma(t))-t\bigr).
\]
\end{definition}




Suppose now that
$\gamma\colon\mathbb{R}\longrightarrow M$
is a line. Its two ends determine Busemann functions
\[
    b^{\pm}(x)
    :=
    \lim_{t\to\infty}
    \bigl(d(x,\gamma(\pm t))-t\bigr)
\]
which satisfy
\[
    b^+(\gamma(s))=-s,
    \qquad
    b^-(\gamma(s))=s.
\]
The triangle inequality gives the fundamental relation
\(
    b^+(x)+b^-(x)\geq 0.
\)

\subsection{The Cheeger--Gromoll splitting theorem}

We recall the classical Riemannian splitting theorem of Cheeger and Gromoll
\cite{CGsplitting}.

\begin{theorem}[Cheeger--Gromoll splitting theorem]
\label{thm:cheeger-gromoll}
Let $(M^n,g)$ be a complete Riemannian manifold satisfying
$
    \Ric_g\geq 0.
$
If $M$ contains a line, then there exists a complete Riemannian manifold
$(N^{n-1},h)$ with $\Ric_h \ge 0$ and an isometry
\[
    (M,g)
    \cong
    \bigl(\mathbb{R}_t \times N,dt^2+h\bigr).
\]
\end{theorem}

The proof proceeds by showing \(
    \Delta b^\pm\leq 0
\) 
in the weak sense by Laplacian comparison
under $\Ric_g\geq 0$. The strong minimum principle therefore yields
$b^+ + b^- \equiv 0$, and hence both $b^+$ and $b^-$ are weakly harmonic.
Elliptic regularity then implies that they are smooth. Moreover,
$|\nabla b^+|\equiv 1$, and the Bochner formula gives
\[
    0
    =
    \frac{1}{2}\Delta |\nabla b^+|^2
    =
    |\nabla^2 b^+|^2
    +
    \Ric_g(\nabla b^+,\nabla b^+),
\]
so $\nabla^2 b^+=0$. Thus $\nabla b^+$ is a parallel unit vector field, and
its flow yields the global isometric splitting where $N=(b^+)^{-1}(0)$.



\subsection{Maximal Euclidean factor}
For a closed Riemannian manifold $(M,g)$ with $\text{Ric}_g\geq 0$, one can apply the splitting theorem iteratively to its universal
cover $(\overline M,\overline g)$, which yields an isometric decomposition
\[
    (\overline M,\overline g)
    \cong
    (\mathbb{R}^k\times N,g_{0}+h).
\]
Here,  $(N,h)$ is complete, simply-connected and contains no line; and $g_0$ denotes the flat Euclidean metric on $\R^k$. Since
$\pi_1(M)$ acts cocompactly on $\overline M$ by deck transformations, the
induced action on the factor $N$ is also cocompact. Hence
\autoref{lem:cocompact-line} implies that $N$ must be compact. The integer
$k$, namely the dimension of the maximal Euclidean factor of $\overline M$,
is uniquely determined by $(M,g)$ and is called the {\bf Euclidean rank} of $\overline M$. Moreover, $\pi_1(M)$ contains $\Z^k$ as a subgroup of finite index. In particular, $\pi_1(M)$ is virtually abelian.


\section{Splitting almost K\"ahler manifolds}\label{section:AKsplitting}


This section is devoted to the proof of \autoref{thm:main}. We write
\(
    \Delta_H:=dd^*+d^*d
\)
for the Hodge Laplacian and use
\(
    \Delta_g f:=\operatorname{div}_g(\nabla f)
\)
for the Laplace--Beltrami operator on functions. Thus
\(
    \Delta_H f=-\Delta_g f
\)
for every smooth function $f$.

\begin{lemma}\label{lem:parallel}
    Let $(X,\omega,J,g)$ be an almost K\"ahler manifold with nonnegative Ricci curvature. Suppose that $V$ is a parallel vector field on $X$. Then $JV$ is parallel.
\end{lemma}

\begin{proof}
    Define $\alpha := \iota_V \omega$. 
    Let $\flat \colon \Gamma(TX) \to \Omega^1(X)$ denote the isomorphism induced by the metric $g$.
    By the compatibility relation $g(-,-) = \omega(-,J-)$, 
    we have \[\alpha = (JV)^\flat.\]
    Furthermore, \[|\alpha|^2 = |(JV)^\flat|^2 = |JV|^2=|V|^2\] is constant. 
    We wish to prove that $\alpha$ is a harmonic 1-form: 
    \begin{equation}\label{eq:alpha_harmonic}
        \Delta_H \alpha = (dd^* + d^* d)\alpha = 0.
    \end{equation}
    Combined with the Bochner--Weitzenb\"ock formula (\cite[Lemma 3.4 and Remark 3.2]{Li2012})
    \begin{equation}\label{eq:Bochner_formula}
        \frac{1}{2} \Delta_g |\alpha|^2 = -\langle\Delta_H\alpha, \alpha\rangle+|\nabla \alpha|^2 + {\Ric}(JV,JV),
    \end{equation}
    we obtain 
    $$ 0 = |\nabla \alpha|^2 + {\Ric}(JV,JV),$$
    and the non-negativity of $\Ric$ forces $\nabla \alpha = 0$. Dualizing, we conclude that $JV$ is parallel.

    Let us prove \eqref{eq:alpha_harmonic}. On an almost K\"ahler manifold, the identity $*\omega = \omega^{n-1}/(n-1)!$ implies
    $$ d^*\omega  = - * d  *\omega = \frac{-1}{(n-1)!} * d(\omega^{n-1}) =0.$$
    Consequently, taking $\{e_j\}$ to be an orthonormal frame field on $X$, we obtain:
    \begin{align*}
        d^*\alpha &= -\sum_j \iota_{e_j} (\nabla_{e_j}\alpha)=\sum_j\big(\alpha(\nabla_{e_j}e_j)-e_j(\alpha(e_j))\big) \\
        &=\sum_j\big(\omega(V,\nabla_{e_j}e_j)-e_j(\omega(V,e_j))\big)=-\sum_j\big((\nabla_{e_j}\omega)(V,e_j)+\omega(\nabla_{e_j}V,e_j)\big)\\
        &= -\sum_j(\nabla_{e_j}\omega)(V,e_j)= \iota_V\big(\sum_j\iota_{e_j}(\nabla_{e_j}\omega)\big)=\iota_V(-d^*\omega) = 0.
    \end{align*}
    This deals with the $dd^*\alpha$ term in \eqref{eq:alpha_harmonic}. For the $d^*d\alpha$ term, we employ Cartan's formula:
    \[
    d\alpha = (d\iota_V + \iota_V d)\omega = \mathscr L_V \omega.
    \] 
   Since $V$ is parallel, it is a Killing vector field.
Consequently, the flow $\Phi_t$ of $V$ consists of
orientation-preserving isometries. The Hodge star is natural under
orientation-preserving isometries, and hence
\[
    \Phi_t^*\circ *=*\circ\Phi_t^*.
\]
Differentiating at $t=0$, we see that $\mathscr{L}_V$ commutes with $*$.
Since $\mathscr L_V$ also commutes with $d$ by Cartan's formula, it commutes with the
codifferential $d^*=\pm * d \, *$. Therefore
\[
    d^*d\alpha
    =
    d^*\mathscr L_V\omega
    =
    \mathscr L_Vd^*\omega
    =
    0.
\]
In total, we see that $(d^*d +dd^*)\alpha=0$.
Therefore $\alpha$ is harmonic, and by the Bochner-Weitzenb\"ock formula \eqref{eq:Bochner_formula}, $\alpha$ is parallel. Hence $JV$ is parallel.
\end{proof}

\begin{remark}
    When $J$ is integrable, \autoref{lem:parallel} is immediate from the fact that $\nabla J = 0$ and the Leibniz rule.
\end{remark}

We now prove our main result, \autoref{thm:main}:

\begin{proof}[Proof of \autoref{thm:main}]
    Applying the Cheeger--Gromoll splitting \autoref{thm:cheeger-gromoll} to the line $\gamma$, we obtain a Riemannian isometry
    \begin{equation}\label{eq:first_splitting}
        (X,g) \cong (\mathbb{R}_s \times Y, ds^2 + h),
    \end{equation}
    where $(Y,h)$ is a Riemannian manifold with nonnegative Ricci curvature. Let $b^+$ denote the Busemann function on $X$ associated to one end of the line $\gamma$, and set $V := \nabla b^+$. Because $V$ is parallel, \autoref{lem:parallel} implies that $J V$ is also  parallel. 

    Denote by $\nabla^X$ and $\nabla^Y$ the Levi-Civita connections on $(X,g)$ and $(Y,h)$ respectively. We claim that $JV$ is a parallel vector field on $(Y,h)$ with respect to $\nabla^Y$. Indeed, since $J$ is compatible with $g$, we have \( g(V,JV)=0. \) As $V$ is sent to $\partial_s$ by the isometry \eqref{eq:first_splitting}, it follows that $JV$ is tangent to the slices $\{s\}\times Y$. Moreover, for every vector field $U$ tangent to $Y$, the Levi-Civita connection of the product metric $ds^2 + h$ satisfies \(\nabla^Y_U(JV)=\nabla^X_U(JV)=0. \) Thus $JV$ is parallel on $(Y,h)$.

    Let $(\overline{Y},\overline{h})$ be the Riemannian universal cover of $(Y,h)$. The lift of $JV$ is again a unit parallel vector field on $(\overline Y,\overline h)$. By the de Rham decomposition theorem (\cite{deRham1952}, \cite[Theorem 10.3.1]{Petersen2016}), there is an isometric splitting 
    \[(\overline{Y},\overline{h})\cong(\R_t\times X',dt^2+g'),\]
    where we choose the coordinate $t$ so that the lift of $JV$ via the universal cover $\pi \colon \overline X \to X$ is identified with $\partial_t$.
   Hence, the Riemannian universal cover of $X$ splits accordingly:
    \begin{equation}\label{eq:second_splitting}
        (\overline X, \overline g) \cong (\R_s\times\R_t \times  X', ds^2 + dt^2 + g').
    \end{equation}
   Under this identification, 
   \[ \partial_s=\pi^*V, \qquad \partial_t=\pi^*(JV), \] and therefore \begin{equation}\label{eq:Jsplit}
       \overline J\partial_s=\partial_t, \qquad \overline J\partial_t=-\partial_s.
   \end{equation}  
    Next we prove that the symplectic and almost complex structures on $\overline X$ split as well. Regarding the symplectic form $\overline \omega$, the compatibility relation $\overline g(-,-) = \overline \omega(-,\overline J-)$ and de Rham isometry \eqref{eq:second_splitting} readily imply
    \begin{equation}\label{eq:ds_dt}
        \iota_{\partial_s}\overline \omega = (\overline{J}\partial_s)^\flat = dt \qand \iota_{\partial_t} \overline \omega = (\overline{J}\partial_t)^\flat = - ds.
    \end{equation}
    In particular, $\overline \omega(\partial_s, \partial_t) = 1$  and $\overline\omega$ has no mixed terms between $\operatorname{span}_{\mathbb R}\{\partial_s,\partial_t\}$ and $TX'$. Moreover, the flows of $\partial_s$ and $\partial_t$ preserve $\overline \omega$, by Cartan's formula and \eqref{eq:ds_dt}. Hence, 
    \begin{equation}\label{eq:omegasplit}
        \overline \omega = ds \wedge dt + \omega', 
    \end{equation} 
    where $\omega'$ is a $2$-form pulled back from a symplectic form on $X'$. 

    Since the metric $\overline{g}$ and the symplectic form $\overline{\omega}$ both split, the almost complex structure $\overline{J}$ also splits as $i\oplus J'$ by the compatibility condition. By \eqref{eq:Jsplit} and \eqref{eq:omegasplit}, $i$ is the standard complex structure on $\C$ and $J'$ is a bundle endomorphism pulled back from an almost complex structure on $X'$. Therefore $(\overline X, \overline \omega, \overline J, \overline g)$ is a split almost K\"ahler manifold.
\end{proof}

    
Applying \autoref{thm:main} repeatedly yields the following structure result.

\begin{corollary}\label{cor:AK_structure}
    Let $(X,\omega,J,g)$ be a closed almost K\"ahler manifold with
    $\Ric_g \geq 0$. If
    \[
    (\overline X, \overline g) \cong (\mathbb R^k\times N,g_{0}+g_N)
    \]
    is its splitting with maximal Euclidean factor $\R^k$ in the Riemannian sense, then 
    $k=2r$ and the splitting is compatible with the almost K\"ahler structure:
    \[
    (\overline X, \overline \omega, \overline J, \overline g) \cong (\C^r\times N, \omega_{0}+ \omega_N, J_0 \oplus J_N, g_{0}+g_N).
    \]
    Here $(N,\omega_N,J_N,g_N)$ is almost K\"ahler and contains no line, and $(\C^r, \omega_0, J_0, g_0)$ is the Euclidean K\"ahler structure on $\C^r$.
\end{corollary}

\begin{remark}
 Applications of the Cheeger--Gromoll structure theorem in the cohomologically symplectic setting were studied by Oprea
 \cite{Oprea02}. In particular, he observed that if \(M^{2n}\) is
a closed c-symplectic manifold with nonnegative Ricci curvature, then its Euclidean rank is even, and the remaining factor $N$ is again c-symplectic. He used this structure to obtain refinements of
 Bochner-type bounds involving the Lusternik--Schnirelmann category.
 More recently, Jauhari--Oprea \cite{JO26} developed
 further Bochner-type results of this kind using distributional category. Our result is instead formulated at the level of the almost K\"ahler form itself.
\end{remark}


\section{Automatic integrability}\label{section:Kahlerness}
In this section, we apply \autoref{thm:main} to derive some criteria for the integrability of the compatible almost complex structure.
\subsection{Dimension four}
We first consider noncompact manifolds.
\begin{lemma}\label{lem:noncompact}
    Let $(X,\omega, J, g)$ be an almost K\"ahler four-manifold such that $g$ is a complete Riemannian metric with $\Ric_g\geq 0$, and $(X,g)$ contains a Riemannian line. Then $J$ is integrable.
\end{lemma}
    \begin{proof}
    Let $(\overline X, \overline \omega, \overline J, \overline  g)$ denote the universal cover with the pullback almost K\"ahler structure. 
    By \autoref{thm:main}, $\overline X \cong \C \times \Sigma$ where $\Sigma$ is a Riemann surface, and $\overline J = i\oplus J_\Sigma$ for some almost complex structure $J_\Sigma$ pulled back from $\Sigma$. Because $\Sigma$ is a Riemann surface, $J_\Sigma$ is integrable. Since the universal cover $\pi \colon \overline X \to X$ is a local diffeomorphism, the vanishing of the Nijenhuis tensor $N_{\overline J} = 0$ implies $N_J = 0$. Hence $J$ must be integrable.
    \end{proof}

Now we address compact manifolds. 

\begin{theorem}\label{thm:dim4}
Let $(X,\omega,J,g)$ be a closed almost K\"ahler four-manifold with $\text{Ric}_g\geq0$. Then one of the following occurs:
\begin{enumerate}[nosep]
    \item $X$ is $\mathbb{CP}^1\times\mathbb{CP}^1$ or  $\mathbb{CP}^2\#k\overline{\mathbb{CP}}{}^2$. $J$ need not be integrable in general.
    \item $X$ is $S^2\times T^2$ or $S^2\tilde{\times}T^2$. In this case, $J$ is integrable.
    \item $X$ is a $K3$ surface, an Enriques surface,
    $T^4$, or a hyperelliptic surface. In this case, $g$ is Ricci-flat and $J$ is integrable.
\end{enumerate}
\end{theorem}

\begin{proof}
Let us first assume $X$ is neither rational nor ruled. Since $\text{Ric}_g\geq 0$ implies the scalar curvature $s_g$ is nonnegative, by \cite{KW75} (see also \cite[Theorem 2.3]{GL80}), if $\text{Ric}_g\neq 0$, then $g$ is conformally equivalent to another metric $g'$ with $s_{g'}>0$. Liu \cite{Liu96} and Ohta--Ono \cite{OhtaOno} showed that if a symplectic four-manifold possesses \emph{any} Riemannian metric of positive scalar curvature, it is diffeomorphic to a rational or ruled manifold. Hence, by our assumption, $\Ric_g=0$. In particular, $g$ is Einstein and Sekigawa's result \cite{Sekigawa} implies that $J$ must be integrable. Moreover, LeBrun \cite[Theorem B]{LeBrun} classified all diffeomorphism types of closed four-manifolds admitting both a symplectic structure and a (possibly unrelated) Einstein metric. In the Ricci-flat case, $X$ must be a $K3$ surface, an Enriques surface,
$T^4$, or a hyperelliptic surface. This gives case (3).

Next, we assume $X$ is rational or ruled. When $X$ is rational, namely $\mathbb{CP}^1\times\mathbb{CP}^1$ or  $\mathbb{CP}^2\#k\overline{\mathbb{CP}}{}^2$,  \autoref{rmk:perturb} demonstrates that $J$ need not be integrable. This gives case (1).

Suppose that $X$ is irrational ruled, \emph{i.e.} the symplectic blowup of an $S^2$-bundle over a Riemann surface of genus $\ge 1$. The ordinary Cheeger--Gromoll splitting \autoref{thm:cheeger-gromoll} implies $\pi_1(X)$ is virtually abelian. So, $X$ cannot be a blowup of an $S^2$-bundle over a Riemann surface of genus $\geq 2$. 

The homotopy long exact sequence of a fibration implies $\pi_1(X) \cong \pi_1(T^2) \cong \Z^2$ and thus $b_1(X)=2$. By the Bochner theorem \cite[Theorem 9.2.3]{Petersen2016}, $X$ admits a nowhere zero parallel vector field, so Poincar\'e--Hopf implies that $X$ has vanishing Euler characteristic. Hence $X$ is minimal. 
To see integrability of $J$, consider the (non-compact) universal cover $(\overline{X},\overline{\omega},\overline{J},\overline{g})$.   \autoref{lem:cocompact-line} guarantees a Riemannian line in $(\overline{X},\overline{g})$, and \autoref{lem:noncompact} then forces integrability of $\overline{J}$. Since $\pi$ is a local diffeomorphism, $J$ is integrable. This gives case (2).
\end{proof}

\begin{remark}
    For $k\geq 9$, we do not know whether $\mathbb{CP}^2\#k\overline{\mathbb{CP}}{}^{2}$ admits an almost K\"ahler metric with nonnegative Ricci curvature and non-integrable $J$. The perturbation argument in \autoref{rmk:perturb} does not apply, since it relies on starting from a K\"ahler metric with positive Ricci curvature.
\end{remark}

  
\subsection{Symplectically aspherical manifolds} 
We now consider symplectically aspherical manifolds in all dimensions.
\begin{corollary}
\label{cor:aspherical_nonnegative_is_kahler-flat}
      Suppose that $(X,\omega,J,g)$ is a closed almost K\"ahler manifold such that $\Ric_g\geq 0$ and $(X,\omega)$ is symplectically aspherical. Then $(X,\omega,J,g)$ is a flat K\"ahler manifold which admits a finite cover by a complex torus.
\end{corollary}
\begin{proof}
    \autoref{cor:AK_structure} provides the  almost K\"ahler splitting 
    $$(\overline X,\overline \omega,\overline J, \overline g) \cong (\C^r\times N, \omega_0 + \omega_N, J_0 \oplus J_N, g_{0} + g_N)$$
    where $N$ is closed. Since $(X,\omega)$ is symplectically aspherical, the pullback $\overline\omega:=\pi^*\omega$ vanishes on $\pi_2(\overline X)$. As $\overline X$ is simply connected, the Hurewicz theorem gives \( \pi_2(\overline X)\cong H_2(\overline X;\Z). \) Hence $\overline\omega$ integrates to zero over all 2-cycles, so \( [\overline\omega]=0 \text{ in }H^2_{\mathrm{dR}}(\overline X;\R) \) by de Rham's isomorphism. In particular, $\overline\omega$ and its pullback $\omega_N \in \Omega^2(N)$ are exact. But $N$ is closed, so non-degeneracy forces $\dim_{\R} N=0$. Therefore $\overline X \cong \C^r$, equipped with its standard flat K\"ahler structure $(\omega_0, J_0,g_0)$. Since the covering map is a local diffeomorphism, $(X,\omega,J,g)$ is flat K\"ahler as well. Finally, since the Euclidean rank of $(X,g)$ is  $2r = \dim_\R X$, the manifold $X$ is a finite quotient of $T^{2r}$ (\cite[Corollary 6.67]{Einsteinbook}).
\end{proof}

\begin{remark}
 The interaction between symplectic asphericity and curvature has also been studied recently by 
Di Cerbo--Dranishnikov--Jauhari \cite[\S5]{CDJ26}.
Among other results, they
prove the non-existence of Riemannian metrics of positive scalar curvature on symplectically aspherical manifolds which also admit a spin finite cover.
 It may be regarded as a higher-dimensional analogue of the Liu--Ohta--Ono theorem in dimension four, invoked in proving \autoref{thm:dim4}.
\end{remark}

\section{Symplectic non-hyperbolicity}\label{section:hyperbolic}
In this section, we use the almost K\"ahler splitting theorem  to relate non-hyperbolicity with nonnegative Ricci curvature.

\subsection{Hyperbolicity and non-hyperbolicity}
 A complex manifold $(X,J)$ is called {\bf Brody hyperbolic} if every holomorphic map $\C\to X$ is constant. 
By Brody's theorem \cite{Brody78}, this definition is equivalent to Kobayashi hyperbolicity for compact complex manifolds, namely, the Kobayashi pseudometric is a metric. 
Hyperbolicity has been extensively studied in complex geometry; see \cite{Kobayashibook,Langbook} and the survey articles \cite{Siu_survey,diverio_survey,Demailly_survey} for further background. 


The notion of Brody hyperbolicity readily extends to almost complex
manifolds: an almost complex manifold $(X,J)$ is called
{\bf hyperbolic} if every $J$-holomorphic map
\(\mathbb C\to X\)
is constant; otherwise, $(X,J)$ is called {\bf non-hyperbolic}. 
Forthcoming work of Yeung \cite{Yeung26} investigates hyperbolicity of $(X,J)$ when the almost complex structure $J$ is a non-integrable perturbation of a complex structure of general type, thereby establishing some analogues of the Green--Griffiths, Lang, and Kobayashi conjectures from the integrable setting. 
For almost complex structures compatible with some symplectic form, 
 Biolley \cite{biolley2004floer} introduced the related notion of almost K\"ahler hyperbolicity and studied its connection with a Floer-theoretic notion of symplectic hyperbolicity. Our almost K\"ahler splitting theorem immediately implies the following almost K\"ahler non-hyperbolicity result.

 \begin{corollary}\label{cor:AKnonhyperbolic}
     Let $(X,\omega,J,g)$ be a closed almost K\"ahler manifold with $\text{Ric}_g\geq 0$. If
     \begin{itemize}[nosep]
         \item $\pi_1(X)$ is infinite;
         \item or $\dim_{\mathbb{R}} X=4$,
     \end{itemize}
     then $(X,J)$ admits a nonconstant $J$-holomorphic map $\mathbb{C}\to X$. 
 \end{corollary}
 \begin{proof}
     Consider the universal cover $(\overline{X},\overline{\omega},\overline{J},\overline{g})$. It suffices to find a nonconstant $\overline J$-holomorphic plane in $(\overline X,\overline J)$, since its projection to $(X,J)$ remains nonconstant. If $\pi_1(X)$ is infinite, then $(\overline X,\overline g)$ contains a Riemannian line by \autoref{lem:cocompact-line}. Hence \autoref{thm:main} applies, and the resulting almost K\"ahler splitting immediately provides a nonconstant $\overline J$-holomorphic plane.

It remains to consider the case $\dim_{\mathbb R}X=4$ with finite $\pi_1(X)$. By \autoref{thm:dim4},
$(\overline X,\overline J)$ is either a complex K3 surface or an almost
complex rational surface. In the K3 case, $\overline J$ is integrable,
and \cite[Corollary 2.2]{KLV14} implies the existence of a nonconstant
entire holomorphic curve. In the rational case, since $\overline J$ is
tamed by the symplectic form $\overline\omega$, there exists a
nonconstant $\overline J$-holomorphic sphere by the symplectic uniruledness of $(\overline{X},\overline{\omega})$ (\cite[Theorem 7.3]{Wendlbook}) and Gromov compactness. Composing such a sphere
with a nonconstant holomorphic map $\mathbb C\to\mathbb{CP}^1$ yields a
nonconstant $\overline J$-holomorphic plane.
 \end{proof}
 
 More generally, one can consider the question of whether a fixed symplectic form $\omega$ imposes hyperbolicity, or non-hyperbolicity, for \emph{every} $\omega$-tame almost complex structure. This leads to the following definition.

\begin{definition}\label{def:hyperbolic}
Let \((X,\omega)\) be a symplectic manifold.
We say that \((X,\omega)\) is {\bf hyperbolic} if, for \emph{every}
\(\omega\)-tame almost complex structure \(J\), every \(J\)-holomorphic
map
\(
\C\to X
\)
is constant.
We say that \((X,\omega)\) is {\bf non-hyperbolic} if, for \emph{every}
\(\omega\)-tame almost complex structure \(J\), there exists a
nonconstant \(J\)-holomorphic map
\(
 \C\to X
\).
\end{definition}

Bangert \cite{Bangert1998Existence} proved that the standard linear symplectic form on the $2n$-dimensional torus $T^{2n}$ is non-hyperbolic. 
Li--Wu \cite{LiWu2012Note} generalized Bangert's theorem to asymptotically standard symplectic forms on $\R^{2n}$.
Recently, Cattalani \cite{spencerhyperbolic} extended non-hyperbolicity to a wider class of symplectic manifolds stabilized by $T^2$, under suitable topological hypotheses. On the hyperbolic side, Chen--Yang \cite{ChenYang} proved that if a compact manifold $X$ is homotopy equivalent to a compact Riemannian manifold with negative sectional curvature, then any symplectic form $\omega$ on $X$ is hyperbolic.

\begin{remark}

A symplectic form is called $\tilde d$-hyperbolic if its lift to the universal cover has a bounded primitive; this notion of hyperbolicity, studied by Polterovich \cite{Polterovichhyperbolic} and K\k{e}dra \cite{Kedrahyperbolic}, mimics Gromov's K\"ahler hyperbolicity \cite{Gromovhyperbolic}.
In fact, $\tilde d$-hyperbolicity implies hyperbolicity in the sense of \autoref{def:hyperbolic}, by \cite[Theorem 4.1]{ChenYang}. (That proof is stated for compatible almost complex structures, but also holds for tamed ones.) 
\end{remark}

\subsection{Bangert's theorem}

Let
\(
\omega_0=\sum_{i=1}^n dx_i\wedge dy_i\) and \(
g_0=\sum_{i=1}^n(dx_i^2+dy_i^2)
\)
be the standard symplectic form and Euclidean metric on
\(\R^{2n}\cong\C^n\).
An almost complex structure \(J\) on \(\R^{2n}\) is called
\emph{bounded} if there exists \(C>0\) such that
\[
|Jv|_{g_0}\leq C|v|_{g_0}
\]
for every tangent vector \(v\). It is called \emph{uniformly tamed by
\(\omega_0\)} if there exists \(\alpha>0\) such that
\[
\omega_0(v,Jv)\geq \alpha |v|_{g_0}^2
\]
for every tangent vector \(v\).

The main ingredient in Bangert's theorem is the following
result.
 
 \begin{proposition}[{\cite[Proposition 2.7]{Bangert1998Existence}}]\label{prop:Bangert}
    Let \(J\) be a bounded almost complex structure on \(\R^{2n}\) which is uniformly tamed by \(\omega_0\). Then there exist numbers \(\rho_j\to\infty\) and \(J\)-holomorphic maps
\(
f_j\colon D(\rho_j)\longrightarrow \R^{2n}
\)
defined on 
\(
D(\rho_j)=\{z\in\C\mid |z|<\rho_j\},
\)
which are uniformly Lipschitz and satisfy
\(
\liminf_{j\to\infty}|f_j'(0)|>0.
\)
 \end{proposition}

Let us briefly recall how \autoref{prop:Bangert} implies Bangert's non-hyperbolicity theorem for the torus. Suppose that the almost complex structure \(J\) on $\R^{2n}$ is the lift of some almost complex structure on \(T^{2n}\) tamed by a linear symplectic form $\omega_{\mathrm{lin}}$ on $T^{2n}\cong \mathbb{R}^{2n}/\mathbb Z^{2n}$. After a linear change of coordinates on $\mathbb{R}^{2n}$, we may assume that the lift of $\omega_{\rm lin}$ is the standard symplectic form $\omega_0$. Under this change of coordinates, the standard lattice is replaced by a rank-$2n$ lattice $\Lambda \subseteq \mathbb{R}^{2n}$. The lifted almost complex structure $J$ is then $\Lambda$-periodic. Since $\mathbb{R}^{2n}/\Lambda$ is compact, $J$ is bounded and uniformly tamed by $\omega_0$.
Hence \autoref{prop:Bangert} applies. After suitable lattice
translations, the resulting \(J\)-holomorphic disks admit a subsequence locally converging to a nonconstant entire \(J\)-holomorphic curve in \(\R^{2n}\). Projecting to the quotient yields Bangert's non-hyperbolicity theorem.

 \begin{theorem}[\cite{Bangert1998Existence}]\label{thm:bangert}
    Any linear symplectic form on $T^{2n}$ is non-hyperbolic.
 \end{theorem}
 
  \begin{remark} 
  Entov--Verbitsky
\cite[Proposition 6.1]{EVLag} (see also \cite{spaceofKahlerform}) showed that every K\"ahler-type symplectic form on \(T^{2n}\) is  diffeomorphic to a linear symplectic form.
Since the non-hyperbolicity property in
\autoref{def:hyperbolic} is invariant under symplectomorphism, it follows that Bangert's \autoref{thm:bangert} can be reformulated as every K\"ahler-type symplectic form on \(T^{2n}\) is non-hyperbolic.
\end{remark}

\subsection{Non-hyperbolicity on symplectically aspherical manifolds}
According to \autoref{cor:aspherical_nonnegative_is_kahler-flat}, a symplectically aspherical manifold $(X,\omega,J,g)$ with $\Ric_g \ge 0$ is covered by a complex torus. For such a manifold $(X,\omega,J,g)$, we now explain how Bangert's theorem implies symplectic non-hyperbolicity.

\begin{corollary*}[\autoref{thm:non-hyperbolicity}]
Let $(X,\omega,J,g)$ be a closed almost K\"ahler manifold satisfying \( \Ric_g\geq0. \) Suppose that $(X,\omega)$ is symplectically aspherical. Then, for every almost complex structure $J'$ tamed by $\omega$, there exists a nonconstant $J'$-holomorphic map \( u\colon \mathbb C\rightarrow X \).
\end{corollary*}

\begin{proof}
    By \autoref{cor:aspherical_nonnegative_is_kahler-flat}, one can identify the universal cover of $(X,\omega,J,g)$ with the standard Euclidean space: \[(\overline{X},\overline{\omega},\overline{J},\overline{g})\cong(\R^{2n},\omega_0,J_0,g_0).\]
    Consider the pullback almost complex structure $\overline{J}':=\pi^*J'$ under the covering map $\pi\colon \overline{X}\rightarrow X$. Since $X$ is compact, $\overline{J}'$ must be bounded and uniformly tamed. Therefore, by \autoref{prop:Bangert}, we can obtain a sequence of uniformly Lipschitz $\overline{J}'$-holomorphic disks satisfying $\liminf_{j\to\infty}|f_j'(0)|>0$. Again, by the compactness of $X$, for each $j$ one can find some $\gamma_j\in\pi_1(X)$ such that $\gamma_j\cdot f_j(0)$ has a subsequence converging to some point in $\R^{2n}$. Denote this subsequence by $h_j$. Since $\pi_1(X)$ acts on $\overline{X}$ both holomorphically and isometrically, $h_j$ is still a sequence of uniformly Lipschitz $\overline{J}'$-holomorphic disks satisfying $\liminf_{j\to\infty}|h_j'(0)|>0$. Therefore, Arzel\'a--Ascoli theorem and Gromov's generalized Weierstrass theorem \cite[III.3.1]{hummel1997} together imply the existence of a $\overline{J}'$-holomorphic entire curve $h\colon \C\to \overline{X}$ (see the paragraph after \cite[Proposition 2.7]{Bangert1998Existence}). Then $u:= \pi\circ h$ is the desired $J'$-holomorphic map to $X$.
\end{proof}

\begin{remark}
   In real dimension four, the manifolds appearing in
\autoref{thm:non-hyperbolicity} are complex tori and hyperelliptic surfaces. In higher dimensions, they are compact flat K\"ahler manifolds admitting a finite cover by a complex torus; these belong to the class of generalized hyperelliptic manifolds considered in \cite{Catanesehyperelliptic}.
\end{remark}

\bibliographystyle{amsalpha}
	\bibliography{main}
    
\end{document}